\documentclass[12pt]{article}
\usepackage{amsmath}
\usepackage{latexsym}
\usepackage{amssymb}
\newtheorem{thm}{Theorem}[section]
\newtheorem{la}[thm]{Lemma}
\newtheorem{Defn}[thm]{Definition}
\newtheorem{Remark}[thm]{Remark}
\newtheorem{prop}[thm]{Proposition}
\newtheorem{cor}[thm]{Corollary}
\newtheorem{Example}[thm]{Example}
\newtheorem{Number}[thm]{\!\!}
\newenvironment{defn}{\begin{Defn}\rm}{\end{Defn}}

\newenvironment{example}{\begin{Example}\rm}{\end{Example}}
\newenvironment{rem}{\begin{Remark}\rm}{\end{Remark}}

\newenvironment{proof}{{\noindent\bf Proof.}}%
                  {\nopagebreak\hspace*{\fill}$\Box$\medskip\medskip\par}   
\newcommand{\Punkt}{\nopagebreak\hspace*{\fill}$\Box$}
\newcommand{\wb}{\overline}
\newcommand{\ve}{\varepsilon}
\newcommand{\at}{\symbol{'100}}
\newcommand{\wt}{\widetilde}
\newcommand{\tensor}{\otimes}
\newcommand{\mto}{\mapsto}
\newcommand{\N}{{\mathbb N}}
\newcommand{\R}{{\mathbb R}}
\newcommand{\K}{{\mathbb K}}
\newcommand{\C}{{\mathbb C}}
\newcommand{\cO}{{\cal O}}
\newcommand{\cg}{{\mathfrak g}}
\newcommand{\dl}{{\displaystyle \lim_{\longrightarrow}}}
\newcommand{\pl}{{\displaystyle \lim_{\longleftarrow}}}
\newcommand{\sub}{\subseteq}
\DeclareMathOperator{\id}{id}
\newcommand{\cT}{{\mathcal T}}
\newcommand{\cY}{{\mathcal Y}}
\DeclareMathOperator{\conv}{conv}
\DeclareMathOperator{\Supp}{supp}
\DeclareMathOperator{\lcx}{lcx}
\begin{document}
$\;$\\[-24mm]
\begin{center}
{\Large\bf Continuity of bilinear maps on direct sums\\[2mm] of topological vector spaces}\\[7mm]
{\bf Helge Gl\"{o}ckner\footnote{Supported by DFG, grant GL 357/5--1.}}\vspace{4mm}
\end{center}
\begin{abstract}\noindent
We prove a criterion for continuity of bilinear maps
on countable direct sums of topological vector spaces.
As a first application, we get a new proof for the fact (due to Hirai et al.\ 2001)
that the map\linebreak
$f\colon C^\infty_c(\R^n)\times C^\infty_c(\R^n)\to C^\infty_c(\R^n)$,
$(\gamma,\eta)\mapsto \gamma*\eta$
taking a pair of
test functions to their convolution is continuous.
The criterion also allows an open problem by K.-H. Neeb to be solved:
If $E$ is a
locally convex space, regard the tensor algebra
$T(E):=\bigoplus_{j\in\N_0}T^j(E)$ as the locally convex direct sum of
projective tensor powers of $E$.
We show that $T(E)$ is a topological algebra if and only if every sequence
of continuous seminorms on~$E$ has an upper bound.
In particular, if~$E$ is metrizable, then $T(E)$ is a topological algebra
if and only if~$ E$ is normable.
Also, $T(E)$ is a topological algebra
if $E$ is DFS or~$k_\omega$.\vspace{3mm}
\end{abstract}
{\footnotesize {\em Classification}:
46M05
(Primary);
42A85,
%
44A35,
%
46A13,
%
46A11,
%
46A16,
%
46E25,
%
46F05\linebreak
%
%
%
%
}\\[6mm]
\noindent
{\bf Introduction and statement of results}\\[2mm]
Consider a bilinear map
$\beta\colon \bigoplus_{i\in\N}E_i\times \bigoplus_{j\in\N}F_j\to H$,
where $H$ is a
topo\-logical vector space
and $(E_i)_{i\in \N}$ and $(F_j)_{j\in \N}$
are sequences of topological vector
spaces (which we identify with the corresponding subspaces of the direct sum).
We prove and exploit the following continuity criterion:\\[4mm]
{\bf Theorem A.}
\emph{$\beta$ is continuous if, for all double sequences $(W_{i,j})_{i,j\in \N}$
of $0$-neighbourhoods in $H$,
there exist $0$-neighbourhoods~$U_i$ and $R_{i,j}$ in~$E_i$
and $0$-neighbourhoods $V_j$ and $S_{i,j}$ in~$F_j$
for $i,j\in\N$, such that}
\begin{equation}\label{goodincl}
\begin{array}{l}
\mbox{$\beta(\hspace*{.4mm}U_i\hspace*{.1mm}\times S_{i,j})\sub W_{i,j}$ for all $i,j\in \N$ such that $i<j$; and}\\
\mbox{$\beta(R_{i,j}\times V_j)\sub W_{i,j}$ for all $i,j\in \N$ such that $i\geq j$.}
\end{array}\vspace{2.1mm}
\end{equation}
As a first application,
we obtain a new proof for the continuity of
the bilinear map
$f\colon C^\infty_c(\R^n)\times C^\infty_c(\R^n)\to C^\infty_c(\R^n)$,
$(\gamma,\eta)\mto \gamma*\eta$
taking a pair of test functions
to their convolution (Corollary~\ref{Hir}).
This was first shown in
\cite[Proposition 2.3]{Hir}.\footnote{For hypocontinuity of convolution
$C^\infty(\R^n)'\!\times\! C^\infty_c(\R^n)\to C^\infty_c(\R^n)$,
see~\cite[p.\,167]{Sw}.}
Our proof allows~$\R^n$ to be replaced with a Lie group $G$, in which case $f$ is continuous
if and only if $G$ is $\sigma$-compact~\cite{Glo}.\\[2.3mm]
For a second application of Theorem~A,
consider a locally convex space~$E$ over $\K\in\{\R,\C\}$.
Let $T^0_\pi(E):=\K$, $T^1_\pi(E):=E$
and endow the tensor powers $T^2_\pi(E):=E\tensor_\pi E$,
$T^{j+1}_\pi(E):=E\tensor_\pi T^j_\pi(E)$ with the projective
tensor product topology (see, e.g., \cite{Sch}).
Topologize the tensor algebra $T_\pi(E):=\bigoplus_{j\in\N_0}T^j_\pi(E)$
(see \cite[XVI, \S7]{LaA})
as the locally convex direct sum~\cite{Bou}.
In infinite-dimensional
Lie theory, the question arose of whether
$T_\pi(E)$ always is a
topological algebra, i.e., whether the algebra multiplication
is continuous~\cite[Problem VIII.5]{Nee}.\footnote{If $\cg$ is a locally convex topological Lie algebra
and $T_\pi(\cg)$ a topological algebra, then also the enveloping algebra $U(\cg)$
(which is a quotient of $T_\pi(\cg)$)
is a topological algebra with the quotient topology.
This topology on $U(\cg)$ has been used implicitly in~\cite{Ne2}.}
We solve this question (in the negative),
and actually obtain a characterization of those locally convex spaces~$E$ for which $T_\pi(E)$
is a topological algebra.\\[2.3mm]
To formulate our solution,
given continuous seminorms~$p$ and~$q$ on~$E$
let us write $p\preceq q$ if $p\leq Cq$ pointwise for some $C>0$.
For~$\theta$ an infinite\linebreak
cardinal number,
let us say that~$E$ satisfies the \emph{upper bound condition} for~$\theta$
(the $\text{UBC}(\theta)$, for short) if
for every set $P$ of continuous seminorms on~$E$ of cardinality $|P|\leq\theta$,
there exists a continuous seminorm~$q$ on~$E$ such that $p\preceq q$ for all $p\in P$.
If $E$ satisfies the $\text{UBC}(\aleph_0)$,
we shall simply say that~$E$ satisfies the \emph{countable upper bound condition}.
Every normable space satisfies the $\text{UBC}(\theta)$,
and there also exist non-normable examples (see Section~\ref{secexs}).
We obtain the following characterization:\\[3.5mm]
{\bf Theorem B.} \emph{Let $E$ be a locally convex space.
Then $T_\pi(E)$ is a topological algebra if and only if $E$
satisfies the countable upper bound condition.}\\[3.5mm]
In particular, for $E$ a metrizable locally convex space,
$T_\pi(E)$ is a topological algebra if and only if~$E$ is normable (Corollary~\ref{thmmetr}).\\[2.3mm]
The upper bound conditions introduced here are also useful
for the theory of vector-valued test functions.
If $E$ is a locally convex space and $M$ a paracompact, non-compact, finite-dimensional
smooth manifold, let $C^\infty_c(M,E)$ be the space of all
compactly supported smooth $E$-valued functions on~$M$.
Consider the bilinear map
\[
\Phi\colon C^\infty_c(M,\R)\times E\to C^\infty_c(M,E), \quad(\gamma,v)\mto \gamma v,
\]
where $(\gamma v)(x):=\gamma(x) v$.
If $M$ is $\sigma$-compact,
then $\Phi$ is continuous if and only if $E$ satisfies
the countable upper bound condition.
If $M$ is not $\sigma$-compact, then $\Phi$ is continuous if and only if $E$ satisfies the $\text{UBC}(\theta)$,
for~$\theta$ the number of connected components of~$M$ (see \cite[Theorem~B]{PRA}).\\[2.3mm]
Without recourse to the countable upper bound condition,
for a certain class of non-metrizable locally convex spaces we show directly that $T_\pi(E)$ is
a topological algebra.
Recall that a Hausdorff topological space $X$ is a \emph{$k_\omega$-space}
if $X=\dl\,K_n$\vspace{-.3mm} as a topological space for a sequence
$K_1\sub K_2\sub\cdots$ of compact spaces (a so-called \emph{$k_\omega$-sequence})
with union $\bigcup_{n=1}^\infty K_n=X$ \cite{Fra}, \cite{GGH}.
For example, the dual space $E'$ of any metrizable locally convex space is a $k_\omega$-space
when equipped with the compact-open topology (cf.\ \cite[Corollary 4.7]{Aus}).
In particular, every Silva space (or DFS-space)
is a $k_\omega$-space, that is, every locally convex direct limit
of Banach spaces $E_1\sub E_2\sub\cdots$,
such that all inclusion maps $E_n\to E_{n+1}$
are compact operators \cite[Example~9.4]{G3}.
For\linebreak
instance,
every vector space of countable dimension (like $\R^{(\N)}$)
is a Silva space (and hence a $k_\omega$-space)
when equipped with
the finest locally convex topology. We show:\\[2.5mm]
{\bf Theorem C.}
\emph{Let $E$ be a locally convex space. If $E$
is a $k_\omega$-space} (\emph{e.g., if $E$ is a DFS-space}),
\emph{then $T_\pi(E)$ is a topological algebra.}\\[2.5mm]
To enable a proof of Theorem~C, we first study tensor powers
$T_\nu^j(E)$ in the category of all (not necessarily locally convex)
topological vector spaces, for~$E$ as in the theorem.\footnote{See \cite{Tu2} and \cite{Ten} for such tensor products,
and the references therein.}
We show that $T^j_\nu(E)$
and $T_\nu(E):=\bigoplus_{j\in \N_0}T^j_\nu(E)$
are $k_\omega$-spaces (Lemmas~\ref{tensko} and~\ref{tensgood})
and that $T_\nu(E)=\dl\,\prod_{j=1}^kT^j_\nu(E)$\vspace{-.3mm}
as a topological space (Lemma~\ref{tensgood}).
This allows us to deduce
that $T_\nu(E)$
is a topological algebra (Proposition~\ref{universe}),
which entails that also the convexification
$T_\pi(E)=(T_\nu(E))_{\lcx }$ is a topological algebra (see Section~\ref{secpc}).\footnote{(Quasi-)convexifications of direct limits of $k_\omega$-spaces also appear in \cite{ATC}, for other goals.}\\[2.5mm]
The conclusion of Theorem~C remains valid if $E=F_{\lcx }$ for a topological\linebreak
vector space $F$ which is a $k_\omega$-space (Proposition~\ref{remvalid}).
This implies, for example, that $T_\pi(E)$ is a topological algebra
whenever $E$ is the free locally convex space over a $k_\omega$-space~$X$ (Corollary~\ref{free}).
Combining this result with Theorem~B,
we deduce: If a locally convex space $E$ is a $k_\omega$-space,
or of the form $E=F_{\lcx }$
for some topological vector space~$F$ which is a $k_\omega$-space, then~$E$ satisfies the
countable upper bound condition (Corollary~\ref{komcount}).\\[2.5mm]
Of course, also many non-metrizable locally convex spaces
$E$ exist for which $T_\pi(E)$ is not a topological algebra.
This happens, for example,
if $E$ has a
topological vector subspace~$F$
which is metrizable but not normable
(e.g., if $E=\R^{(\N)}\times \R^\N$).
In fact, $E$ cannot satisfy the countable upper bound
condition because this property would be inherited by~$F$
\cite[Proposition~3.1\,(c)]{PRA}.
\section{Notational conventions}
Throughout the article, $\K\in \{\R,\C\}$, and
topological vector spaces over~$\K$ are considered
(which need not be Hausdorff).
If $q$ is a seminorm
on a vector space $E$,
we write $B^q_r(x):=\{y\in E\colon q(y-x)<r\}$
and $\wb{B}^q_r(x):=\{y\in E\colon q(y-x)\leq r\}$
for the open (resp., closed) ball of radius $r>0$ around $x\in E$.
We let $(E_q,\|.\|_q)$ be the normed space associated with~$q$,
defined via
\begin{equation}\label{assno}
E_q:=E/q^{-1}(0)\quad\mbox{ and }\quad
\|x+q^{-1}(0)\|_q:=q(x)\,.
\end{equation}
Also, we let
\begin{equation}\label{canoq}
\rho_q\colon E\to E_q\,,\quad \rho_q(x):=x+q^{-1}(0)
\end{equation}
be the canonical map.
If~$q$ is continuous with respect to
a locally convex vector topology on~$E$,
then~$\rho_q$ is continuous.
If $(E,\|.\|)$ is a normed space and $q=\|.\|$,
we also write $B^E_r(x):=B^q_r(x)$ and
$\wb{B}^E_r(x):=\wb{B}^q_r(x)$ for the balls.
A subset $U$ of a vector space $E$ is called \emph{balanced}
if $\wb{B}^\K_1(0) U\sub U$.
We set $\N:=\{1,2,\ldots\}$ and $\N_0:=\N\cup\{0\}$.\\[2.3mm]
If $I$ is a countable set and $(E_i)_{i\in I}$
a family of topological  vector spaces,
its direct sum is the space $\bigoplus_{i\in I}E_i$
of all $(x_i)_{i\in I}\in \prod_{i\in I}E_i$
such that $x_i=0$ for all but finitely many $i\in I$.
The sets of the form
\[
\bigoplus_{i\in I} U_i:=\bigoplus_{i\in I}E_i\cap\prod_{i\in I}U_i,
\]
for $U_i$ ranging through the set of $0$-neighbourhoods
in $E_i$,
form a basis of $0$-neighbourhoods for
a vector topology on $\bigoplus_{i\in I}E_i$.
We shall always equip countable direct sums with this topology
(called the `box topology'), which is locally convex if so is each $E_i$.
Then a linear map $\bigoplus_{i\in I}E_i\to F$
to a topological vector space $F$ is continuous
if and only if all of its restrictions to the~$E_i$
are continuous
(\cite[\S4.1 \& \S4.3]{Jar};
cf.\  \cite{Bou}
for the locally convex case).\\[2.5mm]
A {\em topological algebra} is a topological vector space $A$,
together with a continuous bilinear map
$A\times A\to A$.
If $A$ is assumed associative or unital,
we shall say so explicitly.
\section{Bilinear maps on direct sums}\label{bilsum}
We now prove Theorem~A.
Afterwards, we discuss the hypotheses of the theorem
and formulate special cases which are easier to apply.\\[2.3mm]
{\bf Proof of Theorem A.}
By Proposition~5 in \cite[Chapter I, \S1, no.\,6]{Bou},
the bilinear map $\beta$ will be continuous if
it is continuous at $(0,0)$.
To verify the latter, let $W_0$ be a $0$-neighbourhood in $H$.
Recursively, pick $0$-neighbourhoods $W_k\sub H$ for $k\in\N$
such that $W_k+W_k\sub W_{k-1}$.
Then
\begin{equation}\label{bucket}
(\forall k \in \N)\;\; W_1+\cdots+W_k\sub W_0\,.
\end{equation}
Let $\sigma\colon \N\times \N\to \N$ be a bijection,
and $W_{i,j}:=W_{\sigma(i,j)}$ for $i,j\in\N$. By (\ref{bucket}),
\begin{equation}\label{bucket2}
\bigcup_{(i,j)\in\Phi}W_{i,j}\sub W_0\;\;\,\mbox{for every finite subset $\Phi\sub\N^2$.}
\end{equation}
For $i,j\in\N$, choose $0$-neighbourhoods $U_i$ and $R_{i,j}$ in $E_i$ and $0$-neighbourhoods
$V_j$ and $S_{i,j}$ in $F_j$
such that (\ref{goodincl}) holds.
For $i \in \N$, the set
$P_i:=U_i\cap \bigcap_{j=1}^iR_{i,j}$ is a $0$-neighbourhood in~$E_i$.
For $j\in\N$, let $Q_j\sub F_j$ be the $0$-neighbourhood
$Q_j:=V_j\cap\bigcap_{i=1}^j S_{i,j}$.
We claim that\vspace{-1mm}
\begin{equation}\label{claimnbh}
(\forall i,j\in \N)\;\; \beta(P_i\times Q_j)\sub W_{i,j}\,.\vspace{-1mm}
\end{equation}
If this is true, then $P:=\bigoplus_{i\in \N}P_i$ is a $0$-neighbourhood in
$\bigoplus_{i\in\N}E_i$ and $Q:=\bigoplus_{j\in \N}Q_j$ a $0$-neighbourhood
in $\bigoplus_{j\in\N}F_j$
such that $\beta(P\times Q)\sub W_0$, as\vspace{-1mm}
\[
\sum_{(i,j)\in\Phi}\beta(P_i\times Q_j)\sub
\sum_{(i,j)\in\Phi}W_{i,j}\sub W_0\vspace{-1mm}
\]
for each finite subset $\Phi\sub \N^2$
(by (\ref{claimnbh}) and (\ref{bucket2})) and therefore
$\beta(P\times Q)=\bigcup_\Phi \sum_{(i,j)\in\Phi}\beta(P_i\times Q_j)
\sub W_0$. Thus continuity of $\beta$ at $(0,0)$
is established, once~(\ref{claimnbh})
is verified. To prove~(\ref{claimnbh}), let $i,j\in \N$.
If $i\geq j$, then $\beta(P_i\times Q_j)\sub \beta(R_{i,j}\times V_j)\sub W_{i,j}$.
If $i<j$, then $\beta(P_i\times Q_j)\sub \beta(U_i\times S_{i,j})\sub W_{i,j}$.\vspace{3mm}\Punkt

\noindent
The criterion from Theorem~A is sufficient, but not necessary for
continuity.
\begin{example}
Let $H:=\R^\N$ be the space of all real-valued sequences,
equipped with the product topology,
and $E_i:=F_i:=H$ for all $i\in \N$.
Then $\R^\N$ is an algebra under the pointwise multiplication
\[
\delta \colon \R^\N\times \R^\N \to \R^\N\,,\quad
\delta((x_i)_{i\in\N}, (y_i)_{i\in \N}):=
(x_i)_{i\in\N}\diamond  (y_i)_{i\in \N}:=
(x_iy_i)_{i\in \N}\, .
\]
We show that the bilinear map
\[
\beta\colon \bigoplus_{i\in \N}E_i\times \bigoplus_{j\in \N}F_j\to H\,,\quad
\beta((f_i)_{i\in \N},(g_j)_{j\in \N}):=\sum_{i,j\in \N} f_i\diamond g_j
\]
is continuous, but does not satisfy the
hypotheses of Theorem~A.\\[2.5mm]
To see this, note that the seminorms
\[
p_n\colon \R^\N\to [0,\infty[\,, \quad p_n((x_i)_{i\in \N}):=
\max\{|x_i|\colon i=1,\ldots, n\}
\]
define the topology on $\R^\N$ for $n\in \N$.
For all $n\in \N$, we have
\begin{equation}\label{pnsubm}
(\forall f,g\in \R^\N)\quad
p_n(f\diamond g)\leq p_n(f)p_n(g)\,,
\end{equation}
entailing that $\delta$ is continuous and thus $\R^\N$ a topological algebra.
Also, if $W\sub H$ is a $0$-neighbourhood, then $B^{p_n}_\ve(0)\sub W$
for some $n\in \N$
and $\ve>0$. Set $Q_i:=B^{p_n}_{2^{-i}\sqrt{\ve}}(0)$
for $i\in \N$.
Then $Q_i\diamond Q_j\sub B^{p_n}_{2^{-i}2^{-j}\ve}(0)$
for all $i,j\in \N$ (by (\ref{pnsubm})), entailing
that $Q:=\bigoplus_{i\in \N}Q_i$ is a zero-neighbourhood
in $\bigoplus_{i\in \N} E_i$
such that $\beta(Q\times Q)\sub\sum_{(i,j)\in \N^2}B^{p_n}_{2^{-i}2^{-j}\ve}(0)\sub B^{p_n}_\ve(0)\sub W$.
Hence $\beta$ is continuous at $(0,0)$ and hence continuous.\\[2.5mm]
On the other hand, let $r,s>0$ and $k,m,n\in \N$.
\begin{equation}\label{contrad}
\mbox{If $k>n$ or $k>m$, then }
\,(\exists f\in B^{p_n}_r(0), g\in B^{p_m}_s(0))\;\; f\diamond g \not\in B^{p_k}_1(0)\,.
\end{equation}
In fact, assume that $k>m$ (the case $k>n$ is similar).
Let $e_k\in \R^\N$ be the sequence whose $k$-th entry is~$1$,
while all others vanish.
Then $f:=\frac{r}{2}e_k\in B^{p_n}_r(0)$,
$g:=\frac{2}{r}e_k\in B^{p_m}_s(0)$
(noting that $p_m(g)=0$ since $k>m$),
and $f\diamond g \not\in B^{p_k}_1(0)$
as $p_k(f\diamond g)=p_k(e_k)=1$.\\[2.5mm]
Now consider the $0$-neighbourhoods $W_{i,j}:=B^{p_{i+j}}_1(0)$ in $H$.
Suppose there are $0$-neighbourhoods
$U_i$, $R_{i,j}$, $V_i$ and $S_{i,j}$ in $\R^\N$
such that (\ref{goodincl}) holds -- this will yield a contradiction.
There is $n\in \N$ and $r>0$ such that $B^{p_n}_r(0)\sub U_1$.
Also, for each $j\in \N$ there are $m_j\in \N$ and $s_j>0$ with
$B^{p_{m_j}}_{s_j}(0)\sub S_{1,j}$.
Then
\[
B^{p_n}_r(0)\diamond B^{p_{m_j}}_{s_j}(0) =
\beta(B^{p_n}_r(0)\times B^{p_{m_j}}_{s_j}(0))
\sub W_{1,j}=B^{p_{1+j}}_1(0)
\]
for all $j\geq 2$, by (\ref{goodincl}). Thus $n\geq 1+j$ for all $j\geq 2$, by (\ref{contrad}).
This is impossible.
\end{example}
Our applications use the following consequence of Theorem~A:
\begin{cor}\label{squeeze}
Let $(E_i)_{i\in \N}$ and $(F_j)_{j\in \N}$
be sequences of topological vector spaces
and $H$ be a topological vector space.
Then a bilinear mapping\linebreak
$\beta\colon \bigoplus_{i\in\N}E_i\times \bigoplus_{j\in\N}F_j\to H$
is continuous if
there exist $0$-neighbourhoods~$U_i$ in~$E_i$
and~$V_j$ in~$F_j$
for $i,j\in\N$, such that {\rm(a)} and {\rm(b)} hold:
\begin{itemize}
\item[\rm(a)]
For all $0$-neighbourhoods $W\sub H$
and $i,j\in\N$,
there exists a\linebreak
$0$-neighbourhood~$S_{i,j}$ in~$F_j$
such that $\beta(U_i\times S_{i,j})\sub W$.
\item[\rm(b)]
For all $0$-neighbourhoods $W\sub H$ and $i,j\in\N$,
there exists a\linebreak
$0$-neighbourhood~$R_{i,j}$ in~$E_i$
such that
$\beta(R_{i,j}\times V_j)\sub W$.
\end{itemize}
\end{cor}
\begin{proof}
Let $(W_{i,j})_{i,j\in\N}$ be a double sequence of $0$-neighbourhoods
in $H$.
For $i,j\in\N$, choose $0$-neighbourhoods $U_i\sub E_i$ and $V_j\sub F_j$
as described in the corollary. Then, by (a) and (b)
(applied with $W=W_{i,j}$),
for all $i,j\in \N$ there exist $0$-neighbourhoods
$R_{i,j}\sub E_i$ and $S_{i,j}\sub F_j$ such that
\[
\beta(U_i\times S_{i,j})\sub W_{i,j}\quad\mbox{and}\quad
\beta(R_{i,j}\times V_j)\sub W_{i,j}\,.
\]
Hence Theorem~A applies.
\end{proof}
The next lemma helps to check the hypotheses of Corollary~\ref{squeeze}
in important cases.
\begin{la}\label{useBan}
Let $E$, $F$ and $H$ be topological vector spaces
and $\beta\colon E\times F\to H$ be bilinear.
Assume $\beta=b \circ (\id_E\times \phi)$ for a continuous linear map $\phi\colon F\to X$
to a normed space $(X,\|.\|)$ and continuous bilinear map
$b \colon E\times X\to H$.
Then $V:=\phi^{-1}(B^X_1(0))$
is a $0$-neighbourhood in $F$ such that, for each $0$-neighbourhood
$W\sub H$, there is a $0$-neighbourhood $R\sub E$
with $\beta(R\!\times \!V)\!\sub\!  W$.
\end{la}
\begin{proof}
Since $b^{-1}(W)$ is a $0$-neighbourhood in $E\times X$,
there exist a $0$-neighbourhood $S\sub E$ and $r>0$ such that
$S\times B^X_r(0)\sub b^{-1}(W)$.
Set $R:=rS$. Using that $b$ is bilinear,
we obtain $\beta(R\times V)\sub b(rS\times B^X_1(0))
=b(S\times rB^X_1(0))=b(S\times B^X_r(0))\sub W$.
\end{proof}
If $F$ is a normed space, we can simply set $X:=F$, $\phi:=\id_F$
and $b:=\beta$ in Lemma~\ref{useBan}, i.e., the conclusion is always guaranteed
then (with $V=B^F_1(0)$).
\begin{cor}\label{normedcase}
Let $(E_i)_{i\in \N}$ and $(F_j)_{j\in \N}$
be sequences of normed spaces, $H$ be a
topological vector space and $\beta_{i,j}\colon E_i\times F_j\to H$
be continuous bilinear maps for $i,j\in\N$.
Then the following bilinear map is continuous:\vspace{-.5mm}
\begin{equation}\label{combimap}
\beta \colon \bigoplus_{i\in\N}E_i\times \bigoplus_{j\in\N}F_j\to H\,,\quad
\beta((x_i)_{i\in\N},(y_j)_{j\in\N}):=
\sum_{(i,j)\in \N^2}\beta_{i,j}(x_i,y_j)\,.\vspace{-2mm}
\end{equation}
\end{cor}
\begin{proof}
Lemma~\ref{useBan} shows that
the hypotheses of Corollary~\ref{squeeze}
are satisfied if we define $U_i$ and $V_j$
as the unit balls,
$U_i:=B^{E_i}_1(0)$ and $V_j:=B^{F_j}_1(0)$.
\end{proof}
If $H$ is locally convex, then Corollary~\ref{normedcase}
also follows from \cite[Corollary~2.1]{Dah}.\linebreak
In the locally convex case,
Theorem~A can be reformulated as follows:
\begin{cor}\label{reallyuse}
Let $(E_i)_{i\in \N}$ and $(F_j)_{j\in \N}$
be sequences of locally convex spaces, $H$ be a locally
convex space and $\beta_{i,j}\colon E_i\times F_j\to H$
be continuous bilinear maps for $i,j\in\N$.
Assume that, for every double sequence $(P_{i,j})_{i,j\in \N}$
of continuous seminorms on~$H$,
there are continuous seminorms~$p_i$ $($for $i\in \N)$ and $p_{i,j}$ on~$E_i$
$($for $i\geq j)$
and continuous seminorms~$q_j$ $($for $j\in \N)$ and $q_{i,j}$ on~$F_j$
$($for $i<j)$, such that:
\begin{itemize}
\item[\rm(a)]
$P_{i,j}(\beta_{i,j}(x,y))\leq p_i(x)q_{i,j}(y)$ for all $i<j$ in~$\N$,
$x\in E_i$, $y\in F_j$; and
\item[\rm(b)]
$P_{i,j}(\beta_{i,j}(x,y))\leq p_{i,j}(x)q_j(y)$ for all $i\geq j$ in $\N$
and all $x\in E_i$, $y\in F_j$.
\end{itemize}
Then the bilinear map $\beta$ described in {\rm(\ref{combimap})} is continuous.
\end{cor}
\begin{proof}
Let $W_{i,j}\sub H$ be $0$-neighbourhoods for $i,j\in\N$.
Then there are continuous seminorms $P_{i,j}$ on $H$ such that
$B^{P_{i,j}}_1(0)\sub W_{i,j}$.
Let $p_i$, $p_{i,j}$, $q_j$ and $q_{i,j}$ be as described in Corollary~\ref{reallyuse}.
Then $U_i:=B^{p_i}_1(0)$ and $R_{i,j}:=B^{p_{i,j}}_1(0)$
are $0$-neighbourhoods in $E_i$. Also,
$V_j:=B^{q_j}_1(0)$ and $S_{i,j}:=B^{q_{i,j}}_1(0)$
are $0$-neighbourhoods in $F_j$.
If $i<j$, $x\in U_i$ and $y\in S_{i,j}$, then $P_{i,j}(\beta_{i,j}(x,y))\leq p_i(x)q_{i,j}(y)<1$,
whence $\beta_{i,j}(x,y)\in B^{P_{i,j}}_1(0)\sub W_{i,j}$
and thus $\beta_{i,j}(U_i\times S_{i,j})\sub W_{i,j}$.
Likewise, $\beta_{i,j}(R_{i,j}\times V_j)\sub W_{i,j}$
if $i\geq j$.
Thus Theorem~A applies.
\end{proof}
Let $G$ be a Lie group, with Haar measure $\mu$.
Let $b \colon E_1\times E_2\to F$ be
a continuous bilinear map between locally convex spaces
(where $F$ is sequentially complete), and $r,s,t\in \N_0\cup\{\infty\}$
such that $t\leq r+s$. Using Corollary~\ref{reallyuse}, it is possible to characterize those
$(G,r,s,t,b)$ for which the convolution map
\[
\beta\colon C^r_c(G,E_1)\times C^s_c(G,E_2)\to C^t_c(G,F)\,, \quad
(\gamma,\eta)\mto \gamma*_b\eta
\]
is continuous,
where $(\gamma*_b\eta)(x):=\int_Gb(\gamma(y),\eta(y^{-1}x))\,d\mu(y)$ (see \cite{Glo}).
\section{Continuity of convolution of test functions}
Using the continuity criterion, we obtain a new proof for~\cite[Proposition~2.3]{Hir}:
\begin{cor}\label{Hir}
The map $C^\infty_c(\R^n)\times C^\infty_c(\R^n)
\to C^\infty_c(\R^n)$,
$(\gamma,\eta)\mapsto \gamma*\eta$ is continuous.
\end{cor}
Before we present the proof, let us fix further notation and recall basic facts.
Given an open set $\Omega\sub \R^n$, $r\in \N_0\cup\{\infty\}$
and a compact set $K\sub \Omega$,
let $C^r_K(\Omega)$ be the space of all $C^r$-functions $\gamma\colon \Omega\to\K$
with support $\text{supp}(\gamma)\sub K$.
Using the partial derivatives
$\partial^\alpha\gamma:=\frac{\partial^\alpha \gamma}{\partial x^\alpha}$
for multi-indices $\alpha=(\alpha_1,\ldots,\alpha_n)\in \N_0^n$ of order
$|\alpha|:=\alpha_1+\cdots+\alpha_n\leq r$
and the supremum norm $\|.\|_\infty$,
we define norms
$\|.\|_k$ on $C^r_K(\Omega)$ for $k\in \N_0$ with $k\leq r$
via\vspace{-1mm}
\[
\|\gamma\|_k:=\max_{|\alpha|\leq k}\|\partial^\alpha\gamma\|_\infty,\vspace{-3mm}
\]
and give $C^r_K(\Omega)$ the locally convex vector topology determined
by these norms.
We give $C^r_c(\Omega)=\bigcup_K C^r_K(\Omega)$
the locally convex direct limit\linebreak
topology,
for $K$ ranging through the set of compact subsets of $\Omega$.
\begin{la}\label{facts}
\begin{itemize}
\item[\rm(a)]
The pointwise multiplication
$C^r_K(\Omega)\times C^r_K(\Omega)\to C^r_K(\Omega)$,
$(\gamma,\eta)\mto \gamma\eta$ is continuous.
\item[\rm(b)]
Let $(h_i)_{i\in \N}$ be a locally finite, smooth partition of
unity\footnote{See, e.g., \cite[Chapter II, \S3, Corollary 3.3]{Lan}.}
on $\Omega$, such that each $h_i$ has compact support $K_i:=\Supp(h_i)\sub\Omega$.
Then the linear map
$\Phi\colon C^r_c(\Omega)\to\bigoplus_{i\in \N} C^r_{K_i}(\Omega)$,
$\gamma\mto (h_i\gamma)_{i\in \N}$ is continuous.
\end{itemize}
\end{la}
\begin{proof}
(a) Let $E:=(1,\ldots,1)\in \R^n$
and $k\in \N_0$ such that $k\leq r$.
By the Leibniz Rule,
$\|\partial^\alpha(\gamma\eta)\|_\infty\leq
\sum_{\beta\leq\alpha}
({\alpha \atop \beta})
\|\partial^\beta\gamma\|_\infty
\|\partial^{\alpha-\beta}\eta\|_\infty$,
using multi-index notation.
Since
$\sum_{\beta\leq\alpha}
({\alpha \atop \beta})=
(E+E)^\alpha
=2^{|\alpha|}\leq 2^k$ if $|\alpha|\leq k$
and $\|\partial^\beta\gamma\|_\infty
\|\partial^{\alpha-\beta}\eta\|_\infty\leq \|\gamma\|_k\|\eta\|_k$,
we deduce that $\|\gamma\eta\|_k\leq 2^k\|\gamma\|_k\|\eta\|_k$.
Hence multiplication is continuous at $(0,0)$ and
hence continuous, being bilinear.

(b) To see that the linear map $\Phi$ is continuous, it suffices to show
that its restriction $\Phi_K$ to $C^r_K(\Omega)$
is continuous, for each compact set $K\sub \Omega$.
As $K$ is compact and $(K_i)_{i\in \N}$
locally finite, the set $F:=\{i\in \N\colon
K \cap K_i\not =\emptyset\}$ is finite.
Because the image of $\Phi_K$
is contained in the subspace $\bigoplus_{i\in F}C^r_{K_i}(\Omega)\cong
\prod_{i\in F}C^r_{K_i}(\Omega)$
of $\bigoplus_{i\in\N}C^r_{K_i}(\Omega)
\cong \bigoplus_{i\in F}C^r_{K_i}(\Omega)\oplus\bigoplus_{i\in \N\setminus F}C^r_{K_i}(\Omega)$,
the map $\Phi_K$
will be continuous if its components
with values in $C^r_{K_i}(\Omega)$ are continuous for all $i\in F$.
But these are the maps $C^r_K(\Omega)\to C^r_{K_i}(\Omega)$, $\gamma\mto h_i \gamma$,
which are continuous as restrictions of the maps
$C^r_{K\cup K_i}(\Omega)\to C^r_{K\cup K_i}(\Omega)$, $\gamma\mto h_i \gamma$,
whose continuity follows from~(a).
\end{proof}
If $\gamma\in C^0_c(\R^n)$ and $\eta\in C^0_c(\R^n)$, it is well-known that
$\gamma*\eta\in C^0_c(\R^n)$, with
$\text{supp}(\gamma*\eta)\sub
\text{supp}(\gamma)+\text{supp}(\eta)$ (see \cite[1.3.11]{Hoe}).
Moreover,
\begin{equation}\label{L1}
\|\gamma*\eta\|_\infty\leq \|\gamma\|_\infty\|\eta\|_{L^1}
\;\mbox{and}\;\,
\|\gamma*\eta\|_\infty\leq \|\gamma\|_{L^1}\|\eta\|_\infty,
\end{equation}
since $|(\gamma*\eta)(x)| \leq \int_{\R^n}|\gamma(y)|\,|\eta(x-y)|\,d\lambda(y)
\leq\|\eta\|_\infty\int_{\R^n}|\gamma(y)|\,d\lambda(y)$.
If $\gamma\in C^\infty_c(\R^n)$ and $\eta\in C^0_c(\R^n)$, then $\gamma*\eta\in C^\infty_c(\R^n)$ and
\begin{equation}\label{commu}
\partial^\alpha(\gamma*\eta)=
(\partial^\alpha \gamma)*\eta
\end{equation}
for all $\alpha\in \N_0^n$ (see 1.3.5 and 1.3.6 in \cite{Hoe}).
Likewise, $\gamma*\eta\in C^\infty_c(\R^n)$ for all
$\gamma\in C^0_c(\R^n)$ and $\eta\in C^\infty_c(\R^n)$,
with
$\partial^\alpha(\gamma*\eta)=
\gamma*\partial^\alpha \eta$.
By (\ref{commu}) and (\ref{L1}),
\begin{equation}\label{combi}
\|\gamma*\eta\|_k\leq \|\gamma\|_k\|\eta\|_{L^1}
\end{equation}
for all $\gamma \in C^\infty_c(\R^n)$, $\eta\in C^0_c(\R^n)$ and $k\in \N_0$.
Likewise,
$\|\eta*\gamma\|_k\leq \|\eta\|_{L^1}\|\gamma\|_k$.
We shall also use the obvious estimate
\begin{equation}\label{L1sup}
\|\eta\|_{L^1}\leq \lambda(\text{supp}(\eta))\|\eta\|_\infty
\quad\mbox{for $\eta \in C^0_c(\R^n)$.}
\end{equation}
Hence, given compact sets $K,L\sub \R^n$,
we have $\|\gamma*\eta\|_k\leq \lambda(L)\|\gamma\|_k\|\eta\|_\infty$
for all $\gamma\in C^\infty_K(\R^n)$, $\eta\in C^0_L(\R^n)$
and $k\in \N_0$.
This entails the first assertion of the next
lemma, and the second can be proved analogously:
\begin{la}\label{basccon}
The following bilinear maps are continuous:
\[
C^\infty_K(\R^n)\times C^0_L(\R^n)\to C^\infty_{K+L}(\R^n),\quad (\gamma,\eta)\mto \gamma * \eta\,\,\;\mbox{and}
\]
\[
C^0_K(\R^n)\times C^\infty_L(\R^n)\to C^\infty_{K+L}(\R^n),\quad (\gamma,\eta)\mto \gamma * \eta\,.\qquad
\]
\end{la}
{\bf Proof of Corollary~\ref{Hir}.}
Choose a locally finite, smooth partition of unity $(h_i)_{i\in\N}$
on $\R^n$ such that each $h_i$ has compact support $K_i:=\text{supp}(h_i)$.
Set $E_i:=F_i:= C^\infty_{K_i}(\R^n)$ for $i\in \N$.
Then $X_i:=(C^0_{K_i}(\R^n),\|.\|_\infty)$
is a normed space and inclusion
$\phi_i\colon E_i\to X_i$, $\gamma\mapsto\gamma$ is continuous linear.
Let
$\beta_{i,j}\colon $\linebreak
$E_i\times E_j\to C^\infty_c(\R^n)$,
$\mu_{i,j}\colon E_i\times X_j\to C^\infty_c(\R^n)$,
and
$\nu_{i,j}\colon X_i\times E_j\to C^\infty_c(\R^n)$
be convolution
$(\gamma,\eta)\mto \gamma*\eta$
for $i,j\in \N$.
Lemma~\ref{basccon} implies that $\beta_{i,j}$, $\mu_{i,j}$, and $\nu_{i,j}$
are continuous bilinear.
Since
\[
\beta_{i,j}=\mu_{i,j}\circ (\id_{E_i}\times \phi_j)=\nu_{i,j}\circ (\phi_i\times \id_{E_j}),
\]
Lemma~\ref{useBan} shows that the bilinear map
$\beta\colon \bigoplus_{i\in \N}E_i\times \bigoplus_{j\in \N}E_j\to C^\infty_c(\R^n)$
from (\ref{combimap})
obtained from the above $\beta_{i,j}$
satisfies the hypotheses of Corollary~\ref{squeeze},
with $U_i:=V_i:=\phi_i^{-1}(B^{X_i}_1(0))$.
Hence $\beta$ is continuous.
But the convolution map $f\colon C^\infty_c(\R^n)\times C^\infty_c(\R^n)\to C^\infty_c(\R^n)$
can be expressed as
\begin{equation}\label{clm}
f\; =\; \beta\circ (\Phi\times \Phi)
\end{equation}
with the continuous linear map $\Phi$ introduced in Lemma~\ref{facts}\,(b)
(for $r=\infty$),
as we shall presently verify. Hence, being a composition of continuous maps, $f$
is continuous. To verify (\ref{clm}), let $\gamma,\eta\in C^\infty_c(\R^n)$.
Since $\gamma$ has compact support, only finitely many terms in the sum
$\gamma=\sum_{i\in\N}h_i\gamma$ are non-zero,
and likewise in $\eta=\sum_{j\in\N} h_j\eta$. Hence
\[
f(\gamma,\eta)=\sum_{(i,j)\in \N^2}f(h_i\gamma,h_j\eta)
=\sum_{(i,j)\in \N^2}\beta_{i,j}(h_i\gamma,h_j\eta)
=\beta((h_i\gamma)_{i\in\N}, (h_j\eta)_{j\in\N}),
\]
which coincides with $(\beta\circ(\Phi\times\Phi))(\gamma,\eta)$. The proof is complete.\,\Punkt
\section{Proof of Theorem B}\label{secprB}
We now prove Theorem~B, and then discuss the case of metrizable spaces. \\[2.3mm]
{\bf Proof of Theorem~B.}
Let
$\beta_{0,i}\colon \K\times T^i_\pi(E)\to T^i_\pi(E)$, $(z,v)\mto zv$
and
$\beta_{i,0}\colon T^i_\pi(E)\times \K\to T^i_\pi(E)$, $\beta_{i,0}(v,z):=zv$
be multiplication with scalars, for $i\in \N_0$.
For $i,j\in \N$, let $\beta_{i,j}\colon T^i_\pi(E)\times T^j_\pi(E)\to T^{i+j}_\pi(E)$
be the bilinear map determined by
\[
\beta_{i,j}(x_1\otimes\cdots \otimes x_i, y_1\otimes\cdots \otimes y_j)=
x_1\otimes\cdots \otimes x_i\otimes y_1\otimes\cdots \otimes y_j
\]
for all $x_1,\ldots, x_i,y_1,\ldots, y_j\in E$.
As we are using projective tensor topologies,
all of the bilinear maps $\beta_{i,j}$, $i,j\in \N_0$ are continuous,
which is well known.\\[2.7mm]
We first consider the special case of a normable space~$E$.
Then the multi\-plication $\beta\colon T_\pi(E)\times T_\pi(E)\to T_\pi(E)$
of the tensor algebra is the map~$\beta$ from~(\ref{combimap}),
hence continuous by Corollary~\ref{normedcase}.\\[2.3mm]
Next, let~$E$ be an arbitrary locally convex space satisfying the countable upper bound condition.
Let $U$ be a $0$-neighbourhood in $T_\pi(E)$. After\linebreak
shrinking~$U$ to a box neighbourhood,
we may assume that
$U=\bigoplus_{j\in \N_0}\wb{B}^{q_j}_1(0)$
for
continuous seminorms~$q_j$ on $T^j_\pi(E)$.
For $j\in \N_0$, let $H_j:=((T^j_\pi(E))_{q_j},\|.\|_{q_j})$
be the normed space associated to~$q_j$,
and $\rho_{q_j}\colon T^j_\pi(E)\to H_j$ the canonical map
(see (\ref{assno}) and (\ref{canoq})).
Let $V:=\bigoplus_{j\in \N_0}\wb{B}^{\|.\|_{q_j}}_1(0)\sub\bigoplus_{j\in \N_0}H_j$.
Then
\[
\rho \colon T_\pi(E)\to\bigoplus_{j\in \N_0}H_j,\quad
\rho((x_j)_{j\in \N_0}):=(\rho_{q_j}(x_j))_{j\in \N_0}
\]
is a continuous linear map, and $\rho^{-1}(V)=U$.
If we can show that $\rho\circ \beta$ is continuous,
then $(\rho\circ \beta)^{-1}(V)=\beta^{-1}(\rho^{-1}(V))=\beta^{-1}(U)$
is a $0$-neighbourhood in $T_\pi(E)\times T_\pi(E)$,
entailing that the bilinear map~$\beta$ is continuous at $0$
and hence continuous.\\[2.3mm]
To this end, recall that
the $j$-linear map $\tau_j\colon E^j \to T^j_\pi(E)$ taking $(v_1,\ldots, v_j)$ to
$v_1\tensor\cdots\tensor v_j$ is continuous, for each $j\in \N$.
Hence, for each $j\in \N$,
there exists a continuous seminorm~$p_j$ on~$E$
such that
\begin{equation}\label{hencfacto}
(\forall v_1,\ldots, v_j\in E)\quad
q_j(\tau_j(v_1,\ldots, v_j))\leq p_j (v_1)\cdots p_j(v_j)\,.
\end{equation}
By the countable upper bound condition,
there exists a continuous seminorm~$q$ on~$E$ such that
$p_j\preceq q$ for all $j\in \N$,
say
\begin{equation}\label{thusfac2}
p_j\leq C_j q
\end{equation}
with $C_j>0$.
We let $(E_q,\|.\|_q)$ be the normed space associated with~$q$,
and $\rho_q \colon E\to E_q$ be the canonical map.
For each $j\in \N$,
consider the map
\[
\tau_j'\colon (E_q)^j\to T^j_\pi(E_q), \quad (v_1,\ldots, v_j)\mto v_1\tensor\cdots\tensor v_j\,,
\]
and the direct product map $(\rho_q)^j=\rho_q\times\cdots\times \rho_q\colon E^j\to (E_q)^j$.
Then\linebreak
$\tau_j'\circ (\rho_q)^j\colon E^j\to T^j_\pi(E_q)$ is continuous $j$-linear,
and hence gives rise to a continuous linear map
$\phi_j:=T^j_\pi(\rho_q)\colon T^j_\pi(E)\to T^j_\pi(E_q)$,
determined by
\begin{equation}\label{phij}
\phi_j\circ \tau_j=\tau_j'\circ (\rho_q)^j\,.
\end{equation}
Also, define $\phi_0:=\id_\K$.
Then the linear map
\begin{equation}\label{defnphi}
\phi:=T_\pi(\rho_q)\colon T_\pi(E)\to T_\pi(E_q)\,,\quad (x_j)_{j\in \N_0}\mto (\phi_j(x_j))_{j\in \N_0}
\end{equation}
is continuous (being continuous on each summand).
For each $j\in \N$, there exists a continuous $j$-linear map
\[
\theta_j\colon (E_q)^j\to T^j_\pi(E)_{q_j}=H_j\quad\mbox{ such that }\quad \theta_j\circ (\rho_q)^j=\rho_{q_j}\circ \tau_j\,,
\]
as follows from~(\ref{hencfacto}) and~(\ref{thusfac2}).
Now the universal property of $T^j_\pi(E_q)$ provides a continuous linear map $\psi_j\colon T^j_\pi(E_q)\to
H_j$, determined by
\begin{equation}\label{ourpsi}
\psi_j\circ \tau_j'=\theta_j\,.
\end{equation}
Define $\psi_0:=\rho_{q_0}\colon \K\to H_0$.
Then the linear map
\begin{equation}\label{defnpsi}
\psi\colon T_\pi(E_q)\to \bigoplus_{j\in \N_0}H_j\,,\quad (x_j)_{j\in \N_0}\mto (\psi_j(x_j))_{j\in \N_0}
\end{equation}
is continuous.
By the special case of normed spaces already discussed,
the algebra multiplication $\beta'\colon T_\pi(E_q)\times T_\pi(E_q)\to T_\pi(E_q)$ is continuous.
We now verify that the diagram
\begin{equation}\label{diagra}
\begin{array}{ccc}
T_\pi(E)\times T_\pi(E) & \stackrel{\rho\circ \beta}{\longrightarrow} & {\displaystyle \bigoplus_{j\in \N_0}H_j}\\[1.4mm]
\;\;\;\;\;\;\;\;\downarrow \phi\times\phi & & \uparrow \psi\\[.4mm]
T_\pi(E_q)\times T_\pi(E_q) & \stackrel{\beta'}{\longrightarrow} & T_\pi(E_q)
\end{array}
\end{equation}
is commutative. If this is true, then $\rho\circ\beta=\psi\circ \beta'\circ (\phi\times\phi)$ is continuous,
which implies the continuity of~$\beta$ (as observed above).
Since both of the maps
$\rho\circ\beta$ and $\psi\circ \beta'\circ (\phi\times\phi)$
are bilinear, it suffices that they coincide on $S\times S$ for a subset $S\sub T_\pi(E)$
which spans $T_\pi(E)$. We choose $S$ as the union of $\K$
and $\bigcup_{j\in \N}\tau_j(E^j)$.
For $i,j\in \N$ and $v_1,\ldots, v_i,w_1,\ldots, w_j\in E$,
we have
\begin{eqnarray*}
\lefteqn{\psi(\beta'(\phi(v_1\tensor\cdots\tensor v_j),\phi(w_1\tensor\cdots\tensor w_j)))}\\
&=& \psi(\beta'(
\rho_q(v_1)\tensor\cdots\tensor \rho_q(v_i),
\rho_q(w_1)\tensor\cdots\tensor \rho_q(w_j)))\\
&=&
\psi_{i+j}(\rho_q(v_1)\tensor\cdots\tensor \rho_q(v_i)\tensor
\rho_q(w_1)\tensor\cdots\tensor \rho_q(w_j))\\
&=&
\theta_{i+j}(\rho_q(v_1),\ldots, \rho_q(v_i),
\rho_q(w_1),\ldots, \rho_q(w_j))\\
&=&\rho_{q_{i+j}}(v_1\tensor\cdots\tensor v_i\tensor w_1\tensor\cdots\tensor w_j)\\
&=&(\rho\circ \beta)(v_1\tensor\cdots\tensor v_i, w_1\tensor\cdots\tensor w_j)\,,
\end{eqnarray*}
as required.
For $x,y\in \K$, we have $\psi(\beta'(\phi(x),\phi(y))=\rho_{q_0}(xy)=\rho(\beta(x,y))$.
For $x\in \K$ and $w_1,\ldots,w_j\in E$, we have
\[
\begin{array}{lcl}
\psi(\beta'(\phi(x),\phi(w_1\tensor\cdots\tensor w_j))) \!\! & = &\!\!
\psi(\beta'(x,\rho_q(w_1)\tensor\cdots\tensor \rho_q(w_j)))\\
\;= \; x\psi(\rho_q(w_1)\tensor\cdots\tensor \rho_q(w_j)) \!\! & = & \!\! x\theta_j(\rho_q(w_1),\ldots,\rho_q(w_j))\\
\; = \; x\rho_{q_j}(w_1 \tensor\cdots\tensor w_j) \! \! & & \hspace*{-18mm}=\; \rho(\beta(x,w_1\tensor\cdots\tensor w_j))\,.
\end{array}
\]
Likewise, $\psi(\beta'(\phi(v_1\tensor\cdots\tensor v_i),\phi(y))=\rho(\beta(v_1\tensor\cdots\tensor v_i,y))$
for $v_1,\ldots, v_i\in E$ and $y\in \K$. Hence (\ref{diagra}) commutes, and hence~$\beta$ is continuous.\\[2.3mm]
\emph{If $T_\pi(E)$ is a topological algebra},
let $(p_j)_{j\in \N}$ be any sequence of continuous seminorms on~$E$.
Omitting only a trivial case, we may assume that~$E\not=\{0\}$.
For each $j\in\N_0$, we then find a continuous seminorm $q_j\not=0$
on $T^j_\pi(E)$. Let $Q_0(x):=|x|$ for $x\in \K$.
For $j\in\N$, let $Q_j$ be a continuous
seminorm on $T^j_\pi(E)=E\tensor_\pi T^{j-1}_\pi(E)$ such that
\[
Q_j(x\tensor y)=p_j(x)q_{j-1}(y)\quad\mbox{for all $x\in E$ and $y\in T^{j-1}_\pi(E)$}
\]
(see, e.g., \cite[III.6.3]{Sch}).
Then
$W:=\bigoplus_{j\in \N_0} \wb{B}^{Q_j}_1(0)$
is a $0$-neighbourhood in $T_\pi(E)=\bigoplus_{j\in \N_0}T^j_\pi(E)$.
Since $\beta$ is assumed continuous,
there exists a box neighbourhood $V\sub T_\pi(E)$,
of the form $V=\bigoplus_{j\in \N_0}V_j$
with $0$-neighbourhoods $V_j\sub T^j_\pi(E)$,
such that $\beta(V\times V)\sub W$
and hence
\begin{equation}\label{givescon}
(\forall j\in \N)\quad  \beta_{1,j-1}(V_1\times V_{j-1})\sub \wb{B}^{Q_j}_1(0)\,.
\end{equation}
For $j\in\N$, pick $x_j\in V_{j-1}\sub T^{j-1}_\pi(E)$ such that $q_{j-1}(x_j)\not=0$.
Then $1\geq Q_j(\beta_{1,j-1}(v,x_j))=Q_j(v\tensor x_j)=p_j(v)q_{j-1}(x_j)$ for all $v\in V_1$.
Hence
\begin{equation}\label{willimpl}
p_j(V_1)\sub [0,1/q_{j-1}(x_j)]
\end{equation}
for all $j\in \N$. Let~$q$ be a continuous seminorm on~$E$ such that
$\wb{B}^q_1(0)\sub V_1$. Then (\ref{willimpl})
implies that $p_j\leq \frac{1}{q_{j-1}(x_j)}q$ for each $j\in \N$,
and thus $p_j\preceq q$.  Hence~$E$ satisfies the countable upper bound \vspace{2mm}condition.\,\Punkt

\noindent
\begin{la}\label{metrcount}
Let $E$ be a
metrizable locally convex space.
Then $E$ satisfies the countable upper bound condition
if and and only if~$E$ is normable.
\end{la}
\begin{proof}
If the topology on~$E$ comes from a norm~$\|.\|$,
then $p\preceq \|.\|$ for each continuous seminorm~$p$ on~$E$,
entailing that~$E$ satisfies the countable upper bound condition
(and $\text{UBC}(\theta)$ for each infinite cardinal~$\theta$).
Conversely, let~$E$ satisfy the countable upper bound condition.
Let \mbox{$p_1\leq p_2\leq\cdots$} be a sequence of seminorms defining the topology of~$E$.
Then there exists a continuous seminorm~$q$ on~$E$ such that
$p_j\preceq q$ for all $j\in \N$, say $p_j\leq C_j q$ with $C_j>0$.
It is clear from this that the balls $B^q_r(0)$ form a basis of $0$-neighbourhoods
in~$E$ for $r>0$. Hence~$q$ is a norm and defines the topology of~$E$. 
\end{proof}
In view of Lemma~\ref{metrcount}, Theorem~B has the following immediate
consequence:
\begin{cor}\label{thmmetr}
Let $E$ be a metrizable locally convex space.
Then $T_\pi(E)$ is a topological algebra if and only if~$E$ is normable.\Punkt
\end{cor}
\section{Tensor products beyond local convexity}
We shall deduce Theorem~C
from new results on tensor
products in the category of
general (not necessarily locally convex)
topological vector spaces.
\begin{defn}
Given
topological vector spaces $E_1,\ldots, E_j$ with $j\geq 2$,
we write $E_1\tensor_\nu \cdots\tensor_\nu E_j$
for the tensor product $E_1\tensor\cdots\tensor E_j$, equipped with the
finest vector topology $\cO_\nu$
making the `universal' $j$-linear map
\begin{equation}\label{deftau}
\tau\colon E_1\times\cdots\times E_j\to E_1\tensor\cdots\tensor E_j,
\; (x_1,\ldots, x_j)\mto x_1\tensor\cdots\tensor x_j
\end{equation}
continuous.
\end{defn}
\begin{rem}\label{elemten}
By definition of $\cO_\nu$,
a linear map $\phi\colon E_1\tensor_\nu \cdots\tensor_\nu E_j\to F$ to
a topological vector space $F$ is continuous if and only if $\phi\circ\tau$
is continuous.
If $E_1,\ldots, E_j$ are Hausdorff, then also $E_1\tensor_\nu \cdots\tensor_\nu E_j$
is Hausdorff:  If $E_1,\ldots, E_j$ are locally convex Hausdorff
or their dual spaces separate points,
this follows from the continuity of the identity map
$E_1\tensor_\nu \cdots\tensor_\nu E_j\!\to\! (E_1)_w\tensor_\pi \cdots\tensor_\pi (E_j)_w$,
using weak topologies.
In general,
the Hausdorff property follows by an induction
from the case $j=2$ in \cite{Tu2} (see \cite[Proposition~1\,(d)]{Ten}).
\end{rem}
\begin{la}\label{5.3}
Let $(E,\cO)$ be a Hausdorff topological vector space and
$K_n\not=\emptyset$ be compact, balanced subsets of $E$
such that $E=\bigcup_{n\in\N}K_n$ and
$K_n+K_n\sub K_{n+1}$
for all $n\in\N$.
Let $\cT$ be the topology on $E$ making it the direct limit
$\dl\,K_n$
as a topological space.\footnote{Thus $U\sub E$ is open if and only if $U\cap K_n$ is
relatively open in $K_n$ for each $n\in \N$.}
Then $(E,\cT)$ is a topological vector space.
\end{la}
\begin{proof}
Consider the continuous
addition map $\alpha\colon (E,\cO) \times (E,\cO)\to (E,\cO)$
and the addition map $\alpha'\colon (E,\cT) \times (E,\cT)\to (E,\cT)$.
Because $K_n+K_n\sub K_{n+1}$ and $\cT$
induces the given topology on $K_{n+1}$,
the restriction
$\alpha'|_{K_n\times K_n}=\alpha|_{K_n\times K_n}
\colon K_n\times K_n\to K_{n+1}\sub (E,\cT)$
is continuous.
Since $(E,\cT)\times (E,\cT)=\dl\, (K_n\times K_n)$\vspace{-.3mm}
as a topological space~\cite[Theorem 4.1]{Hir},
we deduce that $\alpha'$ is continuous as a map
$(E,\cT)\times (E,\cT)\to (E,\cT)$.
Next, consider the continuous scalar multiplication $\mu\colon \K\times (E,\cO)\to (E,\cO)$
and the scalar multiplication $\mu'\colon \K\times (E,\cT)\to (E,\cT)$.
To see that $\mu'$ is continuous,
it suffices to show that its restriction to a map
$\wb{B}^\K_{2^j}(0) \times (E,\cT)
\to (E,\cT)$
is continuous for each $j\in\N$.
Since $\wb{B}^\K_{2^j}(0) \times (E,\cT)=
\dl\,\wb{B}^\K_{2^j}(0) \times K_n$\vspace{-.8mm} as
a topological space, we need only show that
the restriction of $\mu'$ to $\wb{B}^\K_{2^j}(0)\times K_n$
is continuous. But $\mu'(\wb{B}^\K_{2^j}(0)\times K_n)=2^j K_n\sub K_{n+j}$,
and $\cT$ induces the given topology on $K_{n+j}$.
Hence $\mu'|_{\wb{B}^\K_{2^j}(0)\times K_n}=\mu|_{\wb{B}^\K_{2^j}(0)\times K_n}$
is continuous.
\end{proof}
\begin{la}\label{tensko}
If the topological vector spaces $E_1,\ldots, E_j$ are $k_\omega$-spaces,
then also
$E_1\tensor_\nu\cdots\tensor_\nu E_j$ is
a $k_\omega$-space.
\end{la}
\begin{proof}
Let $\cO_\nu$ be the topology on
$E:=
E_1\tensor_\nu\cdots\tensor_\nu E_j$.
For $i\in \{1,\ldots, j\}$, pick a $k_\omega$-sequence
$(K_{i,n})_{n\in\N}$ for~$E_i$.
After replacing $K_{i,n}$ with $\wb{B}^\K_1(0) K_{i,n}$,
we may assume that each $K_{i,n}$ is balanced.
Let
\[
K_n :=\sum_{i=1}^{2^n} (K_{1,n}\tensor\cdots \tensor K_{j,n}).
\]
Then each $K_n$ is a compact, balanced subset of
$E$,
and $E=\bigcup_{n\in\N} K_n$.
Since $K_n+ K_n\sub K_{n+1}$ by definition,
Lemma~\ref{5.3} shows that the topology $\cT$ making~$E$
the direct limit topological space $\dl\,K_n$\vspace{-.3mm} is
a vector topology.
As the inclusion maps $K_n\to (E,\cO_\nu)$ are continuous,
it follows that $\cO_\nu\sub \cT$.
Note that~$\tau$ from (\ref{deftau})
maps $L_n:=K_{1,n}\times\cdots \times K_{j,n}$
into $K_n$.
Since $\cT$ and $\cO_\nu$
induce the same topology on $K_n$
and $\tau$ is continuous as a map to $(E,\cO_\nu)$,
it follows that each restriction $\tau|_{L_n}\colon L_n\to K_n\sub (E,\cT)$
is continuous. Thus $\tau$ is continuous to $(E,\cT)$
(as $E_1\times \cdots\times E_j=\dl\, L_n$\vspace{-.3mm} by \cite[Theorem 4.1]{Hir})
and hence $\cT\sub \cO_\nu$.
Thus $\cO_\nu=\cT$,
whence $E$ is the $k_\omega$-space
$\dl\,K_n$.\vspace{-4mm}
\end{proof}
\begin{la}\label{preass}
Consider  topological vector spaces $E_1,\ldots, E_i$ and $F_1,\ldots, F_j$,
$E:=E_1\tensor_\nu \cdots\tensor_\nu E_i$ and $F:=F_1\tensor_\nu \cdots\tensor_\nu  F_j$,
and the bilinear map
\[
\kappa \colon E\times F\to E_1\tensor_\nu \cdots\tensor_\nu E_i\tensor_\nu F_1\tensor_\nu \cdots\tensor_\nu
F_j=:H
\]
determined by $\kappa(x_1\otimes\cdots \otimes x_i, y_1\otimes\cdots \otimes y_j)=
x_1\otimes\cdots \otimes x_i\otimes y_1\otimes\cdots \otimes y_j$.
If $E_1,\ldots, E_i, F_1,\ldots, F_j$ are $k_\omega$-spaces,
then $\kappa$ is continuous and the linear map
\begin{equation}\label{almasso}
\tilde{\kappa}\colon E\tensor_\nu F\to H\;\;\mbox{determined by $\; \tilde{\kappa}(v\tensor w)=\kappa(v,w)$}
\end{equation}
is an
isomorphism of topological vector spaces.
\end{la}
\begin{proof}
Let $\!\tau\colon \!E_1\times \cdots\times E_i\!\to\! E$,
$\! \tau'\colon \! F_1\times\cdots\times F_j\!\to\!  F$,
$\!\tilde{\tau}\colon \! E\times F\! \to\!  E\tensor_\nu F$\linebreak
and
$\tau''\colon E_1\times \cdots\times E_i\times F_1\times\cdots\times F_j\to H$
be the universal maps.
It is known from abstract algebra that $\tilde{\kappa}$ is an isomorphism
of vector spaces. Moreover,
$\tilde{\kappa}^{-1}\circ\tau''
=\tilde{\tau}\circ (\tau\times \tau')$ is continuous,
whence $\tilde{\kappa}^{-1}$ is continuous (see Remark~\ref{elemten}).
Thus $\wt{\kappa}$ will be a topological isomorphism
if $\tilde{\kappa}$ is continuous,
which will be the case if we can show that $\kappa$ is continuous,
as $\tilde{\kappa}\circ \tilde{\tau}=\kappa$
(see Remark~\ref{elemten}).
To this end,
pick $k_\omega$-sequences $(K_{a,n})_{n\in\N}$
and $(K_{b,n}')_{n\in\N}$ of balanced sets for the spaces $E_a$ and $F_b$, respectively.
Then $K_n:=\sum_{k=1}^{2^n} (K_{1,n}\tensor\cdots \tensor K_{i,n})$ and
$K_n':=\sum_{k=1}^{2^n} K_{1,n}'\tensor\cdots \tensor K_{j,n}'$
define $k_\omega$-sequences
$(K_n)_{n\in\N}$ and $(K'_n)_{n\in\N}$
for $E$ and $F$, respectively
(see proof of Lemma~\ref{tensko}).
Moreover, $(K_n\times K_n')_{n\in\N}$ is a $k_\omega$-sequence
for $E\times F$ (cf.\ \cite[Theorem 4.1]{Hir}),
entailing that $\kappa$ will be continuous
if we can show that $\kappa|_{K_n\times K_n'}$ is continuous
for each $n\in\N$.
Consider the map
\[
q_n\colon (K_{1,n}\times\cdots\times K_{i,n}
\times
K_{1,n}'\times\cdots\times K_{j,n}')^{2^n}\to
K_n\times K_n'\, ,
\]
$(x_{1,k},\ldots,x_{i,k},y_{1,k},\ldots,y_{j,k})_{k=1}^{2^n}\!\mto\! (\sum_{k=1}^{2^n}
x_{1,k}\!\tensor \!\cdots\!\tensor \! x_{i,k},\sum_{\ell=1}^{2^n}y_{1,\ell}\!\tensor\! \cdots\!\tensor \!y_{j,\ell})$.
Then $q_n$ is a continuous map from a compact space
onto a Hausdorff space and hence a topological quotient map.
Hence $\kappa|_{K_n\times K_n'}$ is continuous
if and only if $\kappa\circ q_n$ is continuous.
But $\kappa\circ q_n$ is the map
taking $(x_{1,k},\ldots,x_{i,k},y_{1,k},\ldots,y_{j,k})_{k=1}^{2^n}$
to $\sum_{k,\ell=1}^{2^n}
x_{1,k}\tensor \cdots\tensor x_{i,k}\tensor y_{1,\ell}\tensor \cdots\tensor y_{j,\ell}$,
and hence continuous (because $\tau''$ is continuous).\vspace{-4mm}
\end{proof}
\begin{rem}
Although $\nu$-tensor products fail to be associative in general~\cite{Ten},
this pathology is absent in the case of $k_\omega$-spaces $E_1$, $E_2$, $E_3$.
In fact,
the natural vector space isomorphism
$(E_1\tensor_\nu E_2)\tensor_\nu E_3\to
E_1\tensor_\nu (E_2\tensor_\nu E_3)$ is an isomorphism
of topological vector spaces in this case as it can be written as a composition
$(E_1\tensor_\nu E_2)\tensor_\nu E_3\to
E_1\tensor_\nu E_2\tensor_\nu E_3 \to E_1\tensor_\nu (E_2\tensor_\nu E_3)$
of isomorphisms of the form discussed in Lemma~\ref{preass}.
\end{rem}
Our next lemma is a special case of \cite[Corollary 5.7]{GGH}.
\begin{la}\label{tensgood}
Let $E$ be a topological vector space.
If $E$ is a $k_\omega$-space,
then the box topology makes $T_\nu(E):=\bigoplus_{j\in\N_0}T^j_\nu(E)$
a $k_\omega$-space,
and $T_\nu(E)=\dl\,\prod_{j=0}^k T^j_\nu(E)$\vspace{-.3mm}
as a topological space.\Punkt
\end{la}
\begin{prop}\label{universe}
Let $E$ be a topological vector space.
If $E$ is a $k_\omega$-space, then
$T_\nu(E)$ is a topological algebra,
which satisfies a universal property:

For every continuous linear map
$\phi\colon E\to A$ to an associative, unital\linebreak
topological algebra~$A$,
there exists a unique continuous homomorphism\linebreak
$\tilde{\phi}\colon T_\nu(E)\to A$
of unital associative algebras
such that $\tilde{\phi}|_E=\phi$.
\end{prop}
\begin{proof}
Define bilinear maps $\beta_{i,j}\colon T^i_\nu(E)\times T^j_\nu(E)\to T^{i+j}_\nu(E)$
for $i,j\in\N_0$
and the algebra multiplication
$\beta\colon T_\nu(E)\times T_\nu(E)\to T_\nu(E)$
as in Section~\ref{secprB}.
Since countable direct limits and
twofold direct products of $k_\omega$-spaces
can be interchanged by \cite[Proposition 4.7]{GGH},
we have $T_\nu(E)\times T_\nu(E)
=\dl\, P_k$\vspace{-.3mm}
as a topological space,
with $P_k:=\prod_{i,j=1}^k T^i_\nu(E)\times T^j_\nu(E)$
for $k\in\N$.
Hence $\beta$ will be continuous
if $\beta|_{P_k}$ is continuous for each $k\in\N$.
But $\beta(x_1,\ldots, x_k, y_1,\ldots, y_k)
=\sum_{i,j=1}^k\beta_{i,j}(x_i,y_j)$ is a continuous function of
$(x_1,\ldots, x_k, y_1,\ldots, y_k)\in P_k$,
because $\beta_{i,j}\colon T^i_\nu(E)\times T^j_\nu(E)\to T^{i+j}_\nu(E)\sub
T_\nu(E)$ is continuous by Lemma~\ref{preass}.
Thus, $T_\nu(E)$ is a topological algebra.
For $\phi$ as described in the proposition,
there is a unique homomorphism
$\tilde{\phi}\colon T_\nu(E)\to A$
of unital associative algebras
such that $\tilde{\phi}|_E=\phi$
(as is well known from abstract algebra).
For $j\in \N$, let $\tau_j\colon E^j\to T^j_\nu(E)$
be the universal $j$-linear map.
By the universal property of the direct sum,
$\tilde{\phi}$ will be continuous
if $\tilde{\phi}|_{T^j_\nu(E)}$
is continuous for each $j\in \N_0$,
which holds if and only if $\tilde{\phi}\circ \tau_j$
is continuous for each $j\in \N$
(continuity is trivial if $j=0$).
But $\tilde{\phi}\circ \tau_j$
is the map $E^j\to A$, $(x_1,\ldots, x_j)\mto \phi(x_1)\cdots \phi(x_j)$,
which indeed is continuous.\vspace{-7mm}
\end{proof}
\section{Observations on convexifications}
Recall that each topological vector space $Y$
admits a finest locally convex topology $\cO_{\text{lcx}}$
which is coarser than the given topology.
We call $Y_{\text{lcx}}:=(Y,{\cal O}_{\text{lcx}})$
the \emph{convexification} of~$Y$.
Convex hulls of $0$-neighbourhoods
in~$Y$ form a basis of $0$-neighbourhoods for
a locally convex vector topology on~$Y$,
and it is clear that this topology coincides with $\cO_{\text{lcx}}$.
\begin{la}\label{5.1}
If $\theta\colon E_1\times \cdots\times E_j\to Z$
is a continuous $j$-linear map between topological
vector spaces, then $\theta$
is also continuous as a mapping from  $(E_1)_{\lcx }\times\cdots\times
(E_j)_{\lcx }$ to $Z_{\lcx }$.
\end{la}
\begin{proof}
If $W\sub Z$ is a $0$-neighbourhood, there are $0$-neighbourhoods
$U_i\sub E_i$ for $i\in\{1,\ldots, j\}$
with $\theta(U_1\times \cdots\times U_j)\sub W$.
If $x=(x_1,\ldots, x_{j-1})$ is an element of $U_1\times\cdots\times U_{j-1}$,
then $\theta(x,U_j)\sub W$ implies $\theta(x,\conv(U_j))\sub \conv(W)$.
Inductively, $\theta(x_1,\ldots,x_{i-1},\conv(U_i)\times\cdots\times  \conv(U_j))
\sub \conv (W)$ for all $i=j,j-1,\ldots, 1$.
Thus
$\theta(\conv(U_1)\times \cdots\times \conv(U_j))
\sub \conv (W)$.
\end{proof}
\begin{la}\label{cfyalg}
If $A$ is a topological algebra,
with multiplication $\theta\colon A\times A\to A$,
then also $(A_{\lcx },\theta)$ is a topological algebra.
\end{la}
\begin{proof}
Apply Lemma~\ref{5.1} to the bilinear map $\theta$.
\end{proof}
\begin{la}\label{nu2pi}
$(E_1\tensor_\nu\cdots\tensor_\nu E_j)_{\lcx }
=(E_1)_{\lcx }\tensor_\pi\cdots\tensor_\pi (E_j)_{\lcx }$,
for all topological vector spaces $E_1,\ldots, E_j$.
In particular, $(T_\nu^j(E))_{\lcx }=T^j_\pi(E_{\lcx })$
for each topological vector space~$E$.
\end{la}
\begin{proof}
Let $\cO_\nu$ be the topology on
$E_1\tensor_\nu\cdots\tensor_\nu E_j$
and $\cO_\pi$ be the topology on
$(E_1)_{\lcx }\tensor_\pi\cdots\tensor_\pi (E_j)_{\lcx }$.
Since~$\cO_\pi$ is locally convex and coarser
than~$\cO_\nu$, it follows that $\cO_\pi\sub (\cO_\nu)_{\lcx }$.
The universal $j$-linear map
$\tau$ from (\ref{deftau})
is continuous as a map
$E_1\times\cdots\times E_j\to E_1\tensor_\nu\cdots\tensor_\nu E_j$
and hence also continuous as a map
$(E_1)_{\lcx }\times\cdots\times (E_j)_{\lcx }\to (E_1\tensor_\nu\cdots\tensor_\nu E_j)_{\lcx }$,
by Lemma~\ref{5.1}. Hence $(\cO_\nu)_{\lcx }\sub \cO_\pi$,
and hence both topologies coincide.
\end{proof}
\begin{la}\label{5.2}
${\displaystyle \Big(\bigoplus_{j\in\N}E_j\Big)_{\lcx }=
\bigoplus_{j\in\N}(E_j)_{\lcx }}$,
for all topological vector spaces~$E_j$.
\end{la}
\begin{proof}
Both spaces coincide
as abstract vector spaces, and the topology on the
right hand side is coarser.
But it is also finer, because
for all balanced $0$-neighbourhoods $U_j\sub E_j$
and $U:=\bigoplus_{j\in\N}U_j$,
we have $\conv(U_j)\sub \conv(U)$ for each $j$ and thus
$\bigoplus_{j\in\N}2^{-j}\conv(U_j)\sub \conv(U)$.\vspace{-7mm}
\end{proof}
\section{Proof of Theorem C}\label{secpc}
Taking $F:=E$, Theorem~C follows from the next result:
\begin{prop}\label{remvalid}
Let $E$ be a locally convex space.
If $E=F_{\lcx }$ for a topological vector space $F$ which is a $k_\omega$-space,
then $T_\pi(E)$ is topological algebra.
\end{prop}
\begin{proof}
By Proposition~\ref{universe}, $T_\nu(F)$ is a topological
algebra. Hence also $(T_\nu(F))_{\lcx }$ is a topological
algebra, by Lemma~\ref{cfyalg}.
But
\[
(T_\nu(F))_{\lcx }=(\bigoplus_{j\in\N_0}T^j_\nu(F))_{\lcx }
=\bigoplus_{j\in\N_0}(T^j_\nu(F))_{\lcx }
=\bigoplus_{j\in\N_0}T^j_\pi(F_{\lcx })
=\bigoplus_{j\in\N_0}T^j_\pi(E)\]
coincides with $T_\pi(E)$
(using Lemma~\ref{5.2} for the second equality and Lemma~\ref{nu2pi}
for the third).
\end{proof}
The notion of a free locally convex space goes back to~\cite{Mar}.
Given a topological space $X$,
let $\K^{(X)}$ be the free vector space over $X$.
Write $V(X)$
for $\K^{(X)}$,
equipped with the finest vector topology making
the canonical map $\eta_X\colon X\to\K^{(X)}$, $x\mto \delta_{x,.}$ continuous.
Write $L(X)$
for $\K^{(X)}$,
equipped with the finest locally vector topology making
$\eta_X$ continuous.
Call $V(X)$ and $L(X)$ the free topological vector space over~$X$,
respectively, the free locally convex space over~$X$.
\begin{cor}\label{free}
Let $E=L(X)$ be the free locally convex space over a $k_\omega$-space~$X$.
Then $T_\pi(E)$ is a topological algebra.
\end{cor}
\begin{proof}
As is clear,
$L(X)=(V(X))_{\lcx }$.
It is well known that $V(X)$ is
$k_\omega$ if so is~$X$ (see, e.g., \cite[Lemma 5.5]{Kyo}).
Hence Proposition~\ref{remvalid} applies.\vspace{-3mm}
\end{proof}
\section{Some spaces with upper bound conditions}\label{secexs}
Recall from the proof of Lemma~\ref{metrcount}
that every normable space satisfies the $\text{UBC}(\theta)$
for each infinite cardinal~$\theta$.
Combining Theorem~B and Proposition~\ref{remvalid}, we obtain
further examples of spaces with upper bound conditions:
\begin{cor}\label{komcount}
Let $E$ be a locally convex space.
If $E$ is a $k_\omega$-space or $E=F_{\lcx }$ for some
topological vector space which is a $k_\omega$-space,
then $E$ satisfies the countable upper bound condition.\Punkt
\end{cor}
Let $\theta$ be an arbitrary infinite cardinal now.
Then there exists a non-normable space satisfying the $\text{UBC}(\theta)$,
but not the $\text{UBC}(\theta')$ for any $\theta'>\theta$:
\begin{example}
Let $X$ be a set of cardinality $|X|>\theta$, and $\cY$ be the
set of all subsets $Y\sub X$ of cardinality $|Y|\leq \theta$.
Let $E:=\ell^\infty(X)$ be the vector space of bounded $\K$-valued
functions on~$X$,
equipped with the (unusual) vector
topology $\cO_\theta$ defined by the seminorms\vspace{-1mm}
\[
\|.\|_Y\colon E\to [0,\infty[\,,\quad
\|\gamma\|_Y:=\sup\{|\gamma(x)|\colon x\in Y\}\vspace{-1mm}
\]
for subsets $Y\in \cY$.
Note that a function $\gamma\colon X\to\K$
is bounded if and only if all of its restrictions
to countable subsets of~$X$ are bounded.
Hence $E$ can be expressed as the projective limit\vspace{-3mm}
\[
\pl_{Y\in \cY}\,(\ell^\infty(Y),\|.\|_\infty)\vspace{-3mm}
\]
of Banach spaces (with the apparent restriction maps as the bonding maps and limit maps),
and thus~$E$ is complete.
For each $Y\in \cY$, we have $Y\not=X$ by reasons of cardinality,
whence an $y\in X\setminus Y$ exists.
Define $\delta_y\colon X\to\K$, $\delta_y(y):=\delta_{x,y}$ using Kronecker's~$\delta$.
Then $\delta_y\not=0$ and $\|\delta_y\|_Y=0$, whence $\|.\|_Y$
is not a norm. As a consequence, $E$ is not normable.
To see that~$E$ satisfies the $\text{UBC}(\theta)$,
let $(p_j)_{j\in J}$ be a family of continuous seminorms on~$E$
such that $|J|\leq\theta$. For each $j\in J$, there exists a subset $Y_j\sub X$
with $|Y_j|\leq\theta$ and $C_j>0$ such that $p_j\leq C_j\|.\|_{Y_j}$.
Set $Y:=\bigcup_{j\in J}Y_j$. Then $|Y|\leq |J|\,\theta\leq\theta\theta=\theta$.
Hence $q:=\|.\|_Y$ is a continuous seminorm on~$E$,
and $p_j\leq C_j q$ for all~$j$.
Finally, let $Z\sub X$ be a subset of cardinality $\theta<|Z|\leq\theta'$.
Suppose we could find a continuous seminorm $p$ on $E$
such that $\|.\|_{\{z\}}\preceq p$ for all $z\in Z$.
We may assume that $p=\|.\|_Y$ for some $Y\in \cY$.
But then $z\in Y$ for all $z\in Z$
and hence $|Y|\geq |Z|>\theta$, contradiction.
\end{example}
{\small
Helge Gl\"{o}ckner, Universit\"{a}t Paderborn, Institut f\"ur Mathematik,\\
Warburger Str.\ 100, 33098 Paderborn, Germany. \,Email: {\tt  glockner\at{}math.upb.de}}

\begin{thebibliography}{99}\itemsep+.3pt
%
\bibitem{ATC}
Ardanza-Trevijano, S. and M.\,J. Chasco,
\emph{The Pontryagin duality of sequential limits of topological Abelian groups},
J. Pure Appl.\ Algebra {\bf 202} (2005), 11--21.
%
\bibitem{Aus} Au\ss{}enhofer, L., \emph{Contributions to the duality of Abelian topological groups
and to the theory of nuclear groups}, Diss.\ Math.\ {\bf 384}, 1999.
%
\bibitem{Glo} Birth, L. and H. Gl\"{o}ckner, \emph{Continuity of convolution of test functions
on Lie groups}, preprint, arXiv:1112.4729v1.
%
\bibitem{Bou} Bourbaki, N., ``Topological Vector Spaces, Chapters 1--5,'' Springer,
1987.
%
\bibitem{Dah}
Dahmen, R.,
\emph{Analytic mappings between LB-spaces and applications in infinite-dimensional Lie theory},
Math.\ Z. {\bf 266:1} (2010), 115--140.
%
\bibitem{Fra}
Franklin, S.\,P. and B.\,V. Smith Thomas,
\emph{A survey of $k_\omega$-spaces},
Topology Proc.\ {\bf 2} (1978), 111--124.
%
\bibitem{Ten} Gl\"{o}ckner, H.,
\emph{Tensor products in the category
of topological vector spaces are not associative},
Comment.\ Math.\ Univ.\ Carolin.\ {\bf 45:4} (2004), 607Ð-614.
%
\bibitem{G3}
Gl\"{o}ckner, H.,
\emph{Direct limits of infinite-dimensional Lie groups compared to direct limits in related categories},
J.\ Funct.\ Anal.\
{\bf 245:1} (2007), 19--61.
%
\bibitem{Kyo} Gl\"{o}ckner, H., \emph{Instructive examples of smooth, complex differentiable and complex analytic
mappings into locally convex spaces}, J. Math.\ Kyoto Univ.\
{\bf 47:3} (2007), 631--642.
%
\bibitem{PRA}
Gl\"{o}ckner, H., \emph{Upper bounds for continuous seminorms
and special properties of bilinear maps}, preprint, arXiv:1112.1824v2.
%
\bibitem{GGH} Gl\"{o}ckner, H., T. Hartnick and R. K\"ohn, \emph{Final group topologies, Kac-Moody groups and Pontryagin duality}, Israel J. Math.\ {\bf 177} (2010), 49--101.
%
\bibitem{Hir} Hirai, T., H. Shimomura, N. Tatsuuma, E. Hirai, \emph{Inductive limits of topologies,
their direct product, and problems related to algebraic structures},
J.\ Math.\ Kyoto Univ.\ {\bf 41:3} (2001), 475--505.
%
\bibitem{Hoe} H\"ormander, L., ``The Analysis of Linear Partial Differential Operators~I,''
Springer, Berlin, 1990.
%
\bibitem{Jar} Jarchow, H., ``Locally Convex Spaces,'' B.G. Teubner, Stuttgart, 1981.
%
\bibitem{Lan} Lang, S. ``Fundamentals of Differential Geometry,''
Springer, New York, 1999.
%
\bibitem{LaA} Lang, S., ``Algebra,'' Springer, New York, 2002.
%
\bibitem{Mar}
Markov, A., \emph{On free topological groups}, C.R.
(Doklady)
Akad.\ Sci.\ URSS
{\bf 31} (1941), 299--301.
%
\bibitem{Nee} Neeb, K.-H., \emph{Towards a Lie theory of locally convex groups},
Japan J.\ Math.~{\bf 1} (2006), 291--468.
%
\bibitem{Ne2} Neeb, K.-H., \emph{On analytic vectors for unitary representations of infinite dimensional Lie groups},
to appear in Ann.\ Inst.\ Fourier (cf.\ arXiv:1002.4792).
%
\bibitem{Sch} Schaefer, H.\,H.
and M.\,P.\ Wolff,          
``Topological Vector Spaces,''
Springer, 1999.
%
\bibitem{Sw}
Schwartz, L.,
``Th\'{e}orie des distributions,''
Hermann, Paris, 1966.
%
\bibitem{Tu2}
Turpin, Ph., \emph{Produits tensoriels d'espaces vectoriels topologiques}, Bull.\ Soc.\ Math.\ Fr.\ {\bf 110}
(1982), 3Ð-13.
%
\end{thebibliography}
\end{document}